\documentclass[twoside,12pt, a4]{article}
\usepackage[all, 2cell, dvips]{xy}
\usepackage{latexsym,amsmath, amsfonts, graphics, epsf, epic}
\usepackage{amsthm}
\usepackage{amssymb}
\usepackage{amsopn}
\usepackage{amscd}
\usepackage{color}

\newtheorem{thm}{Theorem}[section]

\newtheorem{lem}[thm]{Lemma}
\newtheorem{prop}[thm]{Proposition}
\newtheorem{defn}[thm]{Definition}
\newtheorem{rmk}[thm]{Remark}

\usepackage{amssymb}

\def\sg{\sigma}

\def\dim{\mbox{\rm dim }}

\def\l.l.o.{\it l.l.o}

\def\eps{\varepsilon}

\def\medno{\medskip\noindent}

\font\amsy=msbm10

\def\chiup{\raise 2pt\hbox{$\chi$}}

\DeclareMathOperator{\spn}{span}
\DeclareMathOperator{\orbit}{orbit}

\begin{document}

\centerline{\bf THE GOLOD SHAFAREVICH COUNTER--EXAMPLE }\par
\centerline{\bf WITHOUT HILBERT SERIES}
\vskip 0.5truecm
%

 \centerline{\bf Alon Regev}\vskip 10pt \centerline{\bf
Department of Mathematicsal Sciences}\par \centerline{\bf Northern
Illinois University }\par \centerline{\bf DeKalb, Illinois
60115}\par \centerline{\it E-Mail:~~regev@math.niu.edu}\vskip 5pt
\bigskip\bigskip

\centerline{\bf Amitai Regev}\vskip 10pt \centerline{\bf
Department of Theoretical Mathematics}\par \centerline{\bf The
Weizmann Institute of Science}\par \centerline{\bf Rehovot 76100,
Israel}\par \centerline{\it
E-Mail:~~amitai.regev@weizmann.ac.il}\vskip 5pt
\bigskip\bigskip

\date{\today}

\noindent Abstract: Let $F$ be an arbitrary field. The
Golod--Shafarevich  example of a finitely generated nil
$F$-algebra which is infinite dimensional -- is revisited. Here we
offer a rather elementary treatment of that example, in which
induction replaces Hilbert series techniques. This note also
contains a detailed exposition of the construction of that
example.

\newpage
\section{Introduction}

Throughout this paper $F$ is a field; $d\ge 2$ ~a natural number,
and $T=F\{x_1,\ldots,x_d\}$ is the associative non-commutative
algebra of the polynomials in $d$ variables. Let $T_{\ge 1}\subset
T$
denote the polynomials with no constant term. If $I\subseteq
T_{\ge 1} $ is a two sided ideal then the quotient algebra $
T_{\ge 1}/I $ is finitey generated.
\begin{thm}\label{main.theorem1}\cite{golodshafarevich}
There exists a homogeneous two sided ideal
 $I\subset T_{\ge 1}$, such that the finitely generated algebra
$A=T_{\ge 1}/I$ is both nil and infinite dimensional.
\end{thm}
\begin{defn}\label{definition1}
Let $H=\{f_1,f_2,\ldots\}$ be a sequence of {\it homogeneous}
polynomials. We assume that all $\deg f_j\ge 2$. Let $r_\ell$ be
the number of elements of $H$ of degree $\ell$. Let $I=I_H$ be the
two sided ideal generated in $T_{\ge 1}$ by $H$. Call  $H$  a
``G.S. sequence" if its ideal $I_H$ and numbers $r_\ell$ satisfy
\begin{enumerate}
\item
For every polynomial $g\in T_{\ge 1}$ there exists $n$ such that
$g^n\in I$.
\item
For some $\eps>0$ satisfying $d-2\eps>1,$ $r_\ell\le\eps
^2(d-2\eps)^{\ell-2}$ for all $\ell\ge2$.
\end{enumerate}
\end{defn}
Theorem \ref{main.theorem1} is a corollary of the following two
theorems.

\begin{thm}\label{main.theorem2}
Let $H\subset T_{\ge 1}$ be a G.S. sequence, then the quotient
algebra $T_{\ge 1}/I_H$ is both nil and infinite dimensional.
\end{thm}

\begin{thm}\label{main.theorem3}
There exist G.S. sequences.
\end{thm}

The first step in proving Theorem \ref{main.theorem2} is Theorem
\ref{golod.theorem} \cite{golod1} of Section \ref{s2},  which
applies to any homogeneous ideal $I\subseteq T_{\ge}$
and which establishes the basic inequality \eqref{eq1} below. This
inequality, together with conditions 2 of Definition
\ref{definition1}, imply Proposition \ref{p1}, which states that
$A=T_{\ge 1}/I$ is infinite dimensional. Note that condition 1 of
Definition \ref{definition1} implies that this algebra $A$ is nil.
In Section \ref{s4} we prove the existence of G.S. sequences, see
Theorem \ref{g.s.exist1}, thus proving Theorem \ref{main.theorem3}
and completing the proof of Theorem \ref{main.theorem1}.

\medskip
All the proofs of Proposition \ref{p1} known to us, apply the
techniques of Hilbert series, see \cite{golod1}, \cite{golod2},
 \cite{golodshafarevich}, \cite{herstein}, \cite{rowen} and \cite{vinberg}. In Section
\ref{s3} below we present a new and rather elementary proof of
that proposition, a proof where Hilbert series considerations  are
replaced by an induction argument. The presentation here of the
Golod-Shafarevich theorem is completely elementary.

\section{The basic inequality}\label{s2}
Let $T_n\subset T$ denote the homogeneous polynomials of total
degree $n$, so $\dim T_n=d^n$ and
$$T=\bigoplus _{n=0}^{\infty}T_n.$$
We denote $$ T_{\ge k}=\bigoplus _{n=k}^{\infty}T_n.
$$ Thus
$T_{\ge 1}\subseteq T$  are the polynomials with no constant term.

\medskip
Let $H=\{f_1,f_2,\ldots\}$ be a sequence of {\it homogeneous}
polynomials. We assume that all\\ $\deg f_j\ge 2$, so $H \subset
T_{\ge 2}$. Let $r_n$ be the number of elements of $H$ of degree
$n$. Thus
\[H=\bigcup_{n=2}^{\infty}H_n,\quad\mbox{where}\quad H_n = \{f_j \mid \deg f_j=n\}\quad\mbox{and}
\quad |H_n|=r_n. \]

Let $R=\spn_F\{f_1,f_2,\ldots\}$, so
$$R=\bigoplus_{n=0}^{\infty}R_n=\bigoplus_{n=2}^{\infty}R_n,\quad\mbox{where}\quad R_n=\spn H_n.$$
 Note that $\dim R_n\le r_n$. Since
all $\deg f_j\ge 2,$ we have $r_0=r_1=0$ and $R\subseteq T_{\ge
2}\subseteq T_{\ge 1}$. Since $T_{\ge 1}=T T_1$, we have
$R\subseteq T T_1$.
\medskip
\noindent Let $I=\langle f_1,f_2,\ldots\rangle$ be the two--sided
ideal generated in $T$ by the sequence $H$, so that
$$I=TRT\subseteq T_{\ge 2}\subseteq T_{\ge 1}.$$ Let $A=T_{\ge
1}/I$ be the corresponding algebra. Since $I$ is generated by
homogeneous polynomials of degrees $\ge 2$,
$$I=\bigoplus_{n=0}^{\infty} I_n=\bigoplus_{n=2}^{\infty} I_n$$
where $I_n=I\cap T_n.$ Let $B_n\subseteq T_n $ be a complement
vector space of $I_n$:
\begin{equation}\label{bn}
T_n=I_n\oplus B_n,
\end{equation}
 and denote $b_n=\dim B_n$. Since
$I_0=I_1=0$, hence $B_0=T_0=F$ and $B_1=T_1$, so $b_0=\dim_F F =
1$ and $b_1=\dim_F T_1=d$. Denote $B=\bigoplus_{n\ge 0} B_n$, so
$$T=I\oplus B.$$

\noindent With these notations we prove

\begin{thm}\label{golod.theorem}\cite{golod1}
Let $d\ge 2$, $T=F\{x_1,\ldots,x_d\}=\bigoplus_nT_n$, $I\subseteq
T$ a homogeneous two-sided ideal, $I=\oplus_nI_n$. Write
$T_n=I_n\oplus B_n$ and let $b_n=\dim B_n$. Recall that
$I=\langle f_1,f_2,\ldots\rangle$ where the $f_j$'s are
homogeneous of degrees $\ge 2$, and let $r_\ell$ be the number of
$f_j$'s of degree $\ell$ (thus $r_0=r_1=0$). Then for all $n\ge2$
\begin{equation} \label {eq1}
b_n\ge db_{n-1}-\sum^{n-2}_{j=0}r_{n-j}b_j
\end{equation}
\end{thm}
\begin{proof} (Following \cite{vinberg}). Recall that
$R=\spn_F\{f_1,f_2,\ldots\}$ and that $T=I\oplus B$. We first show
that
\begin{equation} \label{eq2}
I=IT_1+BR.
\end{equation}
Note that
$T=T_{\ge 1}\oplus F=TT_1\oplus F$, hence
\begin{equation} \label{eq3}
I=TRT=TR(TT_1\oplus F)=(TRT)T_1+TR=IT_1+TR.
\end{equation}
Now $T=I\oplus B$, $R\subseteq TT_1$ and $IT=I$, hence
\begin{equation}\label{eq4}
TR=(I\oplus B)R=IR+BR\subseteq ITT_1+BR=IT_1+BR.
\end{equation}
Since $IT_1+IT_1=IT_1$, by \eqref{eq3} and \eqref{eq4}
\begin{equation}\label{eq5}
I=IT_1+TR \subseteq IT_1+(IT_1+BR)\subseteq IT_1+BR.
\end{equation}
Since $I \supseteq IT_1,BR$, we conclude that $I=IT_1+BR$, proving
\eqref{eq2}.

\bigskip \noindent Taking the $n$--th homogeneous component
of \eqref{eq2} yields
\begin{equation}\label{eq6}
I_n=I_{n-1}T_1+\sum_{k=0}^nB_{n-k}R_k=I_{n-1}T_1+\sum_{k=2}^nB_{n-k}R_k.
\end{equation}
Note that $\dim (I_{n-1} T_1) \le (\dim I_{n-1}) (\dim{T_1}) =
(\dim I_{n-1}) d$ and $\dim (B_{n-k} R_k) \le b_{n-k} r_k$.
 Taking dimensions on both sides of \eqref{eq6} we
obtain
\begin{equation}\label{eq7}
\dim I_n\le (\dim I_{n-1})d+\sum_{k=2}^nb_{n-k}r_k.
\end{equation}
Substituting $j=n-k$ in the above sum gives
\begin{equation}\label{eq8}
\dim I_n\le (\dim I_{n-1})d+\sum_{j=0}^{n-2}b_j r_{n-j}
\end{equation}

By \eqref{bn}, $d^n=\dim T_n = \dim I_n + \dim B_n = \dim I_n +
b_n$, so $\dim I_n=d^n-b_n$ and similarly $\dim
I_{n-1}=d^{n-1}-b_{n-1}$. Substituting this into \eqref{eq8}
yields the desired inequality eq1.
\end{proof}

\section{Infinite dimensionality of the algebra $A=T_{\ge
1}/I$}\label{s3}
 At the heart of the Golod-Shafarevich construction is the
 following proposition.
\begin{prop} \label{p1}\cite{golod1}
Let $b_n,\;r_\ell$ be the sequences in Theorem
\ref{golod.theorem}, hence in particular satisfying \eqref{eq1}.
Let $\eps>0$ such that $d-2\eps>1$. If $r_\ell\le\eps
^2(d-2\eps)^{\ell-2}$ for all $\ell\ge2$, then all $b_n\ge1$ (In
fact, it follows that the $b_n$'s grow exponentially), and
therefore the algebra $A=T_{\ge 1}/I$ is infinite-dimensional.
\end{prop}

The proofs of Proposition \ref{p1} known to us, all apply the
techniques of Hilbert series ( see \cite{golod1}, \cite{golod2},
 \cite{golodshafarevich}, \cite{herstein}, \cite{rowen} and \cite{vinberg}). In this note we
deduce Proposition \ref{p1} from Proposition \ref{p2}. We then
give Proposition \ref{p2}
 a rather elementary proof, in which
induction replaces Hilbert series techniques

\begin{prop}\label{p2}
Let $b_n,\;r_\ell$ be the sequences in Theorem
\ref{golod.theorem}, hence in particular satisfying \eqref{eq1}.
Let $v>0$ be a real number satisfying the following condition:

There exist real numbers $c, u>0$ satisfying

\medskip
 (a) For all $n\ge0$, ~~$r_{n+2}\le cu^n$, and

\medskip

(b) $\frac{vd-c}{v+u}\ge v$~~~(so in particular $vd>c$)

\medskip

Then, for all $n\ge0$,
\begin{equation}\label{eq10}
b_n\ge(d-v)^n.
\end{equation}
%
\end{prop}

\noindent Before proving Proposition \ref{p2} we show that it
implies Proposition \ref{p1}. We want to prove that under the
assumptions of  Proposition \ref{p1}, for all $n$, $b_n \ge 1$. By
the assumptions of Proposition \ref{p1}, $r_k\le\eps
^2(d-2\eps)^{k-2}$ for all $k\ge2$. Choose
$v=\eps$, $c=\eps^2$ and $u=d-2\eps$. Then $r_{n+2}\le cu^n$ for
all $n\ge0$, and
$$ \frac{vd-c}{ v+u}~=~\frac{\eps(d-\eps)}{d-\eps}~=~\eps=v~. $$
\medskip
Thus $v$ satisfies the hypothesis of Proposition \ref{p2}, so
$b_n\ge (d-v)^n$, and $(d-v)^n\ge 1$ since $d-v=d-\eps>d-2\eps>1$.
\medskip

%
\bigskip
{\bf The proof of Proposition \ref{p2}.}
\begin{proof}
 We first prove by induction that for all $n\ge 0$,
\begin{equation}\label{eq11}
 vb_{n+1}\ge\sum^{n}_{j=0}cu^{n-j}b_j~.
\end{equation}

\medskip
\underline{$n=0$:}~~Check that $vb_1\ge cb_0$ ~namely, that
~$vd\ge c$, which is given by assumption (b).

\medskip
\underline{\it The inductive step:}~~Assume \eqref{eq11} holds for
some $n$ and show it holds for $n+1$, namely show
$$ vb_{n+2}\ge\sum^{n+1}_{j=0}cu^{n+1-j}b_j~. $$
By the induction hypothesis (which is \eqref{eq11}), combined with
assumption (b),
$$ \frac{vd-c}{v+u}\cdot b_{n+1}\ge
vb_{n+1}\ge\sum^n_{j=0} cu^{n-j}b_j~. $$
By (a) all $cu^{n-j}\ge r_{n+2-j}$, hence
$$ (vd-c)b_{n+1}\ge\sum^n_{j=0}(vcu^{n-j}+ucu^{n-j})
b_j\ge
\sum^n_{j=0}(vr_{n+2-j}+ucu^{n-j})b_j~.$$
This implies that
$$vdb_{n+1}\ge v\sum_{j=0}^n r_{n+2-j}b_j+u\sum_{j=0}^n cu^{n-j}b_j
+ cb_{n+1},$$ so

\[vdb_{n+1}-v\sum^n_{j=0}r_{n+2-j} b_j\ge
u\sum^n_{j=0}cu^{n-j}b_j+cb_{n+1}.\]

Combined with \eqref{eq1} (with $n+2$ replacing $n$, and
multiplying by $v$) we have
$$ vb_{n+2}\ge
~vdb_{n+1}-v\sum^n_{j=0}r_{n+2-j} b_j\ge
u\sum^n_{j=0}cu^{n-j}b_j+cb_{n+1}=\sum^{n+1}_{j=0}cu^{n+1-j}b_j~,
$$
that is,
$$ vb_{n+2}\ge \sum^{n+1}_{j=0}cu^{n+1-j}b_j~. $$
This completes the inductive step and proves \eqref{eq11} for all
$n\ge 0$.

\medskip
Note that by \eqref{eq11} and by assumption (a),
\[vb_{n+1}\ge\sum_{j=0}^nc\cdot u^{n-j}\cdot
b_j\ge\sum_{j=0}^nr_{n+2-j}b_j,\] and therefore

\begin{equation}\label{eq12}
vb_{n+1}\ge\sum_{j=0}^nr_{n+2-j}b_j.
\end{equation}

\bigskip

\noindent  We now show that \eqref{eq1} and \eqref{eq12} together
imply that for all $n$,
\begin{equation}\label{missed1} b_{n+2}\ge
(d-v)b_{n+1},\end{equation} which implies \eqref{eq10}. By
\eqref{eq12}
$$ db_{n+1}-(d-v)b_{n+1}=vb_{n+1}\ge
~\sum^{n}_{j=0}
r_{n+2-j}b_j~, $$ or equivalently
$$ db_{n+1}-\sum^{n}_{j=0}
r_{n+2-j}b_j\ge (d-v)b_{n+1}.
$$
Together with  \eqref{eq1} this implies that
%
$$ b_{n+2}\ge
~db_{n+1}-\sum^{n}_{j=0}r_{n+2-j}b_j\ge
~(d-v)b_{n+1}~. $$ This proves \eqref{missed1} and completes the
proof of Proposition \ref{p2}.
\end{proof}

\bigskip

\section{ The construction of G.S. sequences}\label{s4}

\medskip

\noindent Fix some $\eps>0$ satisfying $d-2\eps>1$. Following
Definition \ref{definition1} we now construct a sequence of
homogeneous polynomials, of degrees $\ge 2$, $f_1,f_2,\ldots\in
T_{\ge 2}\subset T_{\ge 1}\subset T=F\{x_1,\ldots,x_d\}$, having
the properties:

\medskip

\noindent ($i$) For every $g\in T_{\ge 1}$ there exist a power $n$
such that $g^n\in I$, where $I=\langle f_1,f_2,\ldots\rangle$ is
the two--sided ideal generated by the $f_j's$, ~~and

\medskip

\noindent ($ii$) Let $r_\ell$ be the number of $f_j's$ of degree
$\ell$ then $r_{\ell+2}\le \varepsilon^2 (d-2\varepsilon)^\ell$.

\medskip
Let $A=T_{\ge 1}/I$, then $A$ is generated by the $d$ elements
$x_1+I,\ldots,x_d+I$, hence is finitely generated. By ($i$) here
it is nil, and by ($ii$) together with Proposition \ref{p1} it is
infinite dimensional.

\subsection{Preparatory remarks}
Let $q$ and $n$ be  positive integers. Define
$I(q,n)=\{(i_1,\ldots,i_n)\mid 1\le i_1,\ldots,i_n\le q\},$ and
define $J(q,n)\subseteq I(q,n)$ via $J(q,n)=\{(i_1,\ldots,i_n)\mid
1\le i_1\le i_2\le\cdots\le i_n\le q\}.$
\begin{rmk}\label{remark2}
 Note that $|I(q,n)|=q^n$. Also, $|J(q,n)|={n+q-1\choose q-1}$; for example,
 the correspondence
 \[1\le j_1\le\cdots\le j_n\le q\quad\longleftrightarrow \quad1\le
 j_1<j_2+1<j_3+2<\cdots<j_n+n-1\le n+q-1\]
 gives a bijection
 between $J(q,n)$ and the set of $n$-subsets of an $n+q-1$-set. Note that
\begin{equation}\label{eq21}
 |J(q,n)|={n+q-1\choose q-1}\le (n+q-1)^{q-1},
 \end{equation}
 and the right hand side is a polynomial in $n$ (of degree $q-1$).
\end{rmk}

Given $\pi\in S_n$ and  $\underline  i=(i_1,\ldots,i_n)\in
I(q,n)$, define $\pi(\underline
i)=(i_{\pi^{-1}(1)},\ldots,i_{\pi^{-1}(n)})$, then
$\sg(\pi(\underline i))=(\sg\pi)(\underline i)$, hence $I(q,n)$ is
the disjoint union of the corresponding orbits. Given $\underline
i\in I(q,n)$, let $O_{\underline  i}=\orbit (\underline  i)=\{\pi
(\underline  i)\mid \pi\in S_n\},$ denote the orbit of
${\underline i}$ under the $S_n$ action. Then
\[I(q,n)=\bigcup_{\underline j\in J(q,n)}O_{\underline j}\;,\] a
disjoint union. If $\underline  j=(j_1,\ldots,j_n)\in J(q,n)$ then
$\underline j =(1,\ldots,1,2,\ldots,2,\ldots )$ and we denote
$\underline j=(1^{\mu_1},\ldots,q^{\mu_q})$, where $k$ appears
$\mu_k$ times in $\underline  j$. Then
$S_{\mu_1}\times\cdots\times S_{\mu_q}$ fixes $\underline  j$, and
\[|O_{\underline j}|=\frac{n!}{\mu_1!\cdots\mu_q!}.\] For example,
$(1,1,1,2,2,2,2)=(1^3,2^4)\in J(2,7)$, and
$|O_{(1^3,2^4)}|=7!/(3!4!)=35.$

\subsection{The order-symmetric polynomials}

\begin{defn}
Let $\underline  j=(j_1,\ldots,j_n)\in J(q,n)$ then define ``the
order-symmetric polynomial"
\[s_ {\underline j}(y_1,\ldots,y_q)=\sum_{\underline  i\in
O_{\underline j}}y_{i_1}\cdots y_{i_n},\quad\mbox{a homogeneous
polynomial of degree $n$ in $y_1,\ldots,y_q$}.
\]

Note that we can also write

\[s_ {\underline
j}(y_1,\ldots,y_q)=\frac{1}{\mu_1!\cdots \mu_q!}\cdot \sum_{\pi\in
S_n}y_{j_{\pi^{-1} (1)}}\cdots y_{j_{\pi^{-1} (n)}}.\]
\end{defn}

Recall that $d\ge 2$. Given $0<c\in \mathbb{N}$, denote
$q=d+d^2+\cdots+d^c$, then $q$ is the number of monomials of
degree between $1$ and $ c$ in the non-commuting variables
$x_1,\ldots,x_d$. Let $\{M_1,\ldots,M_q\}$ be the set of these
monomials. Given $0<n\in \mathbb{N}$, and $\underline  j\in
J(q,n)$, denote
\begin{equation}\label{eq16}h_{\underline  j}(x)=h_{\underline
j}(x_1,\ldots,x_d)=s_ {\underline  j}(M_1,\ldots,M_q).
\end{equation}

 With these notations we prove
\begin{lem}\label{lemma3.2}
Let $0<c\in \mathbb{N}$,  $q=d+d^2+\cdots+d^c$, with
$\{M_1,\ldots,M_q\}$ the monomials of degrees between 1 and $c$.
 Let $0<n\in \mathbb{N}$ and let
$I\subseteq F\{x_1\ldots,x_d\}$ be a two-sided ideal containing
$h_{\underline  j}(x)$ for all
 $\underline  j\in J(q,n)$,
 where $h_{\underline  j}(x)$ are given by \eqref{eq16}.
 Let $g=g(x_1,\ldots,x_d)\in T_{\ge 1}$ be a polynomial of degree $\le c$.
 Then $g^n\in I$.
\end{lem}
\begin{proof}
Since $M_1,\ldots,M_q$ are all the non-constant monomials (in
$x_1,\ldots,x_d$) of degrees $\le c$, and since $\deg g\le c$, we
can write
\[g=\sum_{i=1}^q\alpha_iM_i\quad\mbox{with}\quad \alpha_i\in F.\]
Then
\[g^n=\sum_{1\le i_1,\ldots,i_n\le q}(\alpha_{i_1}M_{i_1})\cdots (\alpha_{i_n}M_{i_n})
=\sum_{1\le i_1,\ldots,i_n\le q}(\alpha_{i_1}\cdots \alpha_{i_n})
(M_{i_1}\cdots M_{i_n}).
\]
Since the $\alpha_i$'s commute, for any $\pi\in S_n$,
$\alpha_{i_1}\cdots \alpha_{i_n}=\alpha_{i_{\pi^{-1}(1)}}\cdots
\alpha_{i_{\pi^{-1}(n)}}$, and it follows that
\[\sum_{1\le i_1,\ldots,i_n\le q}(\alpha_{i_1}\cdots \alpha_{i_n})
(M_{i_1}\cdots M_{i_n})=\sum_{\underline  j\in
J(q,n)}(\alpha_{j_1}\cdots \alpha_{j_n})\cdot \sum_{\underline
i\in O_{\underline j}}M_{i_1}\cdots M_{i_n}= \]
\[=\sum_{\underline  j\in J(q,n)}(\alpha_{j_1}\cdots
\alpha_{j_n})\cdot s_ {\underline
j}(M_1,\ldots,M_q)=\sum_{\underline j\in
J(q,n)}(\alpha_{j_1}\cdots \alpha_{j_n})\cdot h_{\underline
j}(x_1,\ldots,x_d)
\]
by \eqref{eq16}. Thus $g^n$ is a linear combination of the
polynomials $\{h_{\underline j}(x_1,\ldots,x_d)\mid \underline
j\in J(q,n)\}$. By assumption all these $h_{\underline  j}(x)$ are
in $I$, hence $g^n\in I$.
\end{proof}


\begin{rmk}\label{remark3}
We saw in Remark \ref{remark2} that $|J(q,n)|\le (n+q-1)^{q-1}.$
The right hand side is a polynomial in $n$ (of degree $q-1$),
hence it grows slower than any exponential function in $n$, in
particular, slower than $\varepsilon^2\cdot\alpha^n$, provided
$\varepsilon>0$ and $\alpha
> 1$. This implies the following:

\medskip
Let $d\ge 2, \varepsilon>0$ such that $d-2\varepsilon>1$, then
there exists $n$ large enough such that
\begin{equation}\label{eq22}|J(q,n)|
\le
\varepsilon^2\cdot(d-2\varepsilon)^{n-2}.\end{equation}
\end{rmk}

\subsection{The construction}
We now prove
\begin{thm}\label{g.s.exist1}
G.S. sequences exist. Namely: let $d\ge 2$, $\eps >o$ such that
$d-2\eps >1$. Then there exist a sequence $f_1,f_2\ldots\in T$ of
homogeneous polynomials of degrees $\ge 2$ satisfying conditions 1
and 2 of Definition \ref{definition1}.
\end{thm}
\begin{proof}
The construction is inductive, starting with the empty sequence.
The induction assumption is that by the $k$-th step we have chosen
integers $c_k\le c'_k$ and a sequence of homogenous polynomials
$f_1,\ldots ,f_{m_k}$ of degrees between $2$ and $ c'_k$. We
assume that $f_1,\ldots ,f_{m_k}$ satisfy the following

\medskip
condition($c_k$):
\begin{enumerate}
\item
For any  $\ell\le c'_k$, the number $r_{\ell}$ of the elements $f_
i$ of degree $\ell$ satisfies $r_{\ell+2}\le
\varepsilon^2\cdot(d-2\varepsilon)^\ell$.
\item
For any $g\in T_{\ge 1}$ of degree $\le c_k$ there exists some
power $n$ such that $g^n\in I_k$, where $I_k=\langle f_
1,\ldots,f_ {m_k}\rangle$.
\end{enumerate}

\medskip
{\bf The inductive step.}

In the next step (step $k+1$) we choose  integers $c_{k+1}\le
c'_{k+1}$ satisfying $c'_{k}< c_{k+1}$, and construct another
block of polynomials $f_{m_k+1},\ldots,f_{m_{k+1}}$ having degrees
$c'_k<\deg f_j<c'_{k+1}$, thus obtaining the sequence $f_1,\ldots,
f_{m_k},f_{m_k+1},\ldots,f_{m_{k+1}}$.

\medskip

The construction starts with choosing any $c_{k+1}>c'_k$. Let
$q=q_{k+1}=d+d^2+\cdots+d^{c_{k+1}}$, then choose $n$ large enough
such that
\begin{equation}\label{eq17} c'_k<n \quad\mbox{and}\quad |J(q,n)|< \eps ^2(d-2\eps)^{n-2}.
\end{equation} By \eqref{eq22} such $n$ exists. We choose
$c'_{k+1}=n\cdot c_{k+1}$. Denote $m_{k+1}=|J(q,n)|+m_k$, then let
$f_ {m_k+1},f_ {m_k+2},\ldots ,f_ {m_{k+1}}$
 be the $|J(q,n)|$ polynomials
\[h_{\underline  j}(x_1,\ldots,x_d)=s_ {\underline  j}(M_1,\ldots,M_q),\quad \underline j\in J(q,n),\]
given by \eqref{eq16}. Note that for $\underline j\in J(q,n)$,
$\deg s_ {\underline j}(y_1,\ldots,y_q)=n$ and all $1\le \deg
M_i\le c_{k+1}$, hence
\begin{equation}\label{ineq2} c'_{k}<n\le
\deg s_ {\underline j}(M_1,\ldots,M_q)\le n\cdot
c_{k+1}=c'_{k+1}.\end{equation}

\medskip
Form now the new sequence from the above two blocks of $f_j's$:
$$f_1,\ldots,f_{m_k},f_{m_k+1},\ldots
,f_{m_{k+1}}.$$ We show that this last sequence satisfies
$condition(c_{k+1})$. We also show  that $r_\ell$ calculated in
step $k$ for $\ell\le c'_k$, remains unchanged when calculated in
step $k+1$. Hence the numbers $r_\ell$ are well defined for the
resulting infinite sequence $f_1,f_2,\ldots$  This will complete
the proof of the theorem.

\medskip
Note that $\langle f_{m_k+1},f_{m_k+2},\ldots
,f_{m_{k+1}}\rangle\subseteq \langle
f_1,\ldots,f_{m_k},f_{m_k+1},\ldots ,f_{m_{k+1}}\rangle=I_{k+1}$.
By Lemma \ref{lemma3.2}, for any polynomial $g\in T_{\ge 1}$ of
degree $\le c_{k+1}$, ~$g^n\in\langle f_{m_k+1},f_{m_k+2},\ldots
,f_{m_{k+1}}\rangle$, hence $g^n\in I_{k+1}$, which is part 2 of
$condition(c_{k+1})$.

\medskip
By $condition(c_k)$ the degrees of the polynomials in the
 block   $f_1,\ldots,f_{m_k}$ are $\le c'_k$. By \eqref{eq17}
$c'_k<n$, and by \eqref{ineq2} the degrees of $f_{m_k+1},\ldots
,f_{m_{k+1}}$ are between $n$ and $c'_{k+1}=c_{k+1}n.$
Recall that for for each degree $\ell$, $r_\ell$ is the number of
elements of degree $\ell$ in the sequence
$f_1,\ldots,f_{m_k},f_{m_k+1},\ldots ,f_{m_{k+1}}$. Thus for
$\ell\le c'_k$ there is no contribution to $r_\ell$ from the
second block.
 Therefore these $r_\ell$'s remain unchanged when the
process of constructing the $f_j$'s continues.

\smallskip

If $\ell\le c'_k$, it is given (by the induction assumption) that
$r_{\ell}\le \varepsilon^2\cdot(d-2\varepsilon)^{\ell-2}$.\\

Note that similarly, for $\ell>c_k$ there is no contribution to
$r_\ell$ from the first block, so $r_\ell$ is calculated on the
second block only.

 If $\ell>c'_k$, we
may assume that $\ell\ge n$ (since there are no $f_j$'s with
degree between $c'_k+1$ and $n$), then by \eqref{eq17}
$$r_\ell\le |J(q,n)|
<\varepsilon^2\cdot(d-2\varepsilon)^{n-2}\le
\varepsilon^2\cdot(d-2\varepsilon)^{\ell-2}.$$ This completes the
inductive step.
\end{proof}

{\bf We summarize:}

We are given $d\ge 2$, $\varepsilon> 0$ such that $d-2
\varepsilon> 1$. We have constructed the infinite sequence
$\{f_1,f_2,\ldots\}$ of homogeneous polynomials in
$x_1,,\ldots,x_d$ of degrees $\ge 2$. Let $I=\langle
f_1,f_2,\ldots\rangle\subseteq F\{x_1,\ldots,x_d\}$ denote the
ideal generated by the $f_j$'s. Let $r_\ell$ denote the number of
$f_j$'s of degree $\ell$.  These polynomials satisfy that
$r_\ell\le \varepsilon^2\cdot(d-2\varepsilon)^{\ell-2}$ for all
$\ell\ge 2$. Then, by Proposition \ref{p1} the algebra $T_{\ge
1}/I$ is infinite dimensional. Obviously, it is finitely
generated. But by construction, for every $g\in T_{\ge 1}$ there
exist $n$ such that $g^n\in I$, hence $T_{\ge 1}/I$ is nil.


\begin{thebibliography}{99}


\bibitem{golod1} E. S. Golod, On nil-algebras and finitely approximable $p$-groups,~~Izv. Akad.
Nauk SSSR Ser. Mat. 28 (1964), 273--276; English transl., Amer.
Math. Soc. Transl. (2) 48 (1965), 103--106. MR 28 \#5082.

\bigskip


\bibitem{golod2}
E. S. Golod, Some problems of Burnside type,~~Proc. I.C.M.,
Moscow, 1966, pp.284--289.
%

\bigskip
\bibitem{golodshafarevich}
E. S. Golod and I. R. \v Safarevi\v c, On class field
towers,~~Izv. Akad. Nauk SSSR Ser. Mat. 28 (1964), 261--272;
English transl., Amer. Math. Soc. Transl. (2) 48 (1965), 91--102.
MR 28 \#5056.
%

\bigskip
\bibitem{herstein}
I. N. Herstein, Topics in Ring Theory,~~University of Chicago
Press, Chicago, (1969).
%

\bigskip
\bibitem{rowen}
L. Rowen,~~Ring Theory, Vol. II,~~Academic Press, (1988).
%

\bigskip
\bibitem{vinberg}
E. B. Vinberg, On a theorem on infinite dimensionality of an
associative algebra,~~Izv. Akad. Nauk SSSR Ser. Mat. 29 (1965)
209--214 (Russian); AMS Translations Ser. 2, Vol. 82 (1969)
237--242 (English).
%
\end{thebibliography}
\end{document}